\documentclass[twocolumn]{autart}

\usepackage[T1]{fontenc}
\usepackage{amsmath}
\usepackage{amsfonts}
\usepackage{amssymb}
\usepackage{lmodern}
\usepackage{booktabs}
\usepackage{systeme}
\usepackage{graphicx}

\newtheorem{theorem}{Theorem}
\newtheorem{proposition}[theorem]{Proposition}

\newtheorem{lemma}[theorem]{Lemma}

\makeatletter

\def\@begintheorem#1#2{%
  \par
  \addvspace{6pt}%
  \noindent
  \textbf{#1~#2.}\quad
  \itshape
}

\def\@opargbegintheorem#1#2#3{%
  \par
  \addvspace{6pt}%
  \noindent
  \textbf{#1~#2 (#3).}\quad
  \itshape
}

\def\@endtheorem{%
  \par
  \addvspace{6pt}%
}

\makeatother

\newenvironment{proof}[1][Proof]{%
  \par\noindent\textit{#1.}\quad
}{%
  \hfill$\square$\par
}

\begin{document}

\begin{frontmatter}

\title{Predictor-Based Stabilization of a Delayed Hetero-Dimensional Shoreline--Bathymetry System}

\thanks[footnoteinfo]{
Corresponding author A. Ciss\'e.}

\author[Paestum]{Amadou Ciss\'e}\ead{amadou.cisse@univ-lorraine.fr},    
\author[Paestum,Rome]{Mohamed Boutayeb}\ead{mohamed.boutayeb@univ-lorraine.fr},               

\address[Paestum]{CRAN-UMR-CNRS-7039 University of Lorraine, F-57000, France}  
\address[Rome]{TicLab Universit\'e Internationale de Rabat--Maroc}             


\begin{abstract}
Exponential stabilization is investigated for a hetero-dimensional parabolic system in which a one-dimensional shoreline subsystem, subject to a constant input delay, is bidirectionally coupled with unactuated two-dimensional bathymetric dynamics.
A spectral decomposition separates the finite-dimensional critical shoreline modes from the exponentially stable residual dynamics.
State feedback is designed on the critical subspace and lifted to the full shoreline space. An auxiliary transport equation represents the delay, while a predictor removes the explicit input delay from the nominal predictor dynamics. Exponential stability of the full coupled system is established by combining shoreline stabilization and bathymetric coercivity using weighted input--output estimates and a small-gain condition on the hetero-dimensional interaction.
Numerical simulations show that uncompensated delayed feedback can strongly amplify both shoreline and bathymetric responses, whereas predictor compensation restores decay of the critical shoreline dynamics and prevents delay-induced amplification from propagating to the unactuated bathymetric subsystem.
\end{abstract}

\end{frontmatter}

\section{Introduction}
\label{sec:introduction}

Coupled reaction--advection--diffusion systems describe interacting processes involving transport and diffusion across distinct spatial domains. Coastal morphodynamics provides a representative setting, where shoreline migration and bathymetric evolution result from interacting hydrodynamic and sediment-transport mechanisms \cite{Ribas2015,Geleynse2015,kim2015h2d}.
Shoreline evolution is represented along a one-dimensional coastal coordinate, whereas bathymetric variations develop over a two-dimensional nearshore domain. PDE systems coupling states defined on domains of different spatial dimensions have been investigated in mixed-dimensional and bulk--surface settings \cite{Kuchta2021Mixed,Elliott2017BulkSurface}.
The bidirectionally coupled shoreline--bathymetry configuration considered here belongs to this broader class of hetero-dimensional interconnections, but differs from standard vector-valued PDE systems posed on a common spatial domain.

Parabolic PDE stabilization has been extensively investigated using boundary or distributed actuation, including backstepping, output-feedback, optimal-control, and adaptive designs \cite{smyshlyaev2005backstepping,vazquez2016boundary,zekraoui2025output,orlov2017output,glowinski2022bilinear,bolognani2008adaptive}.
Partial actuation introduces an additional difficulty, as the unactuated dynamics must be stabilized indirectly through the coupling. Exponential stabilization results have been obtained for coupled reaction--diffusion systems under partial actuation \cite{kitsos2023stabilization}.
These results concern coupled states evolving on a common spatial domain and do not directly address the hetero-dimensional shoreline--bathymetry coupling considered here.

Input delays constitute a further difficulty in distributed-parameter control, as they may arise from actuator dynamics, transport mechanisms, or communication constraints and can substantially degrade closed-loop stability. Predictor-based compensation represents the input delay as an augmented distributed state, commonly through an auxiliary transport equation, and constructs a predictor state for the nominal delay-free feedback design.
Such techniques have been developed for parabolic PDEs and extended to reaction--advection--diffusion systems with distributed actuation \cite{krstic2009compensating,qi2019stabilization,wang2021delay}.
Related stability and robustness results for delayed distributed and infinite-dimensional systems can be found in \cite{fridman2014introduction,kang2025delay}.
These predictor-based developments, however, do not directly account for the hetero-dimensional and underactuated coupling considered here.

The problem combines hetero-dimensional coupling, partial actuation, unstable shoreline modes, and a constant input delay. Control is restricted to the one-dimensional shoreline, while the two dimensional bathymetry is influenced indirectly through the bidirectional coupling. Stabilization must therefore compensate for the delayed shoreline actuation and account for the interaction between the actuated and unactuated dynamics. The coupling operators connect states defined on domains of different spatial dimensions, and their contributions to the energy balance cannot be assumed to cancel or provide dissipation. Existing mixed-dimensional PDE studies and predictor-based stabilization results do not directly address this combination.

The main contributions to the stabilization of the delayed hetero-dimensional shoreline--bathymetry system are:
\begin{itemize}
    \item \textbf{Hetero-dimensional morphodynamic model.}
    A coupled parabolic model is formulated to describe the bidirectional interaction between a one-dimensional shoreline and a two-dimensional bathymetric subsystem. Distributed actuation is restricted to the shoreline, while the bathymetry remains unactuated.

    \item \textbf{Spectral feedback design.}
    A spectral decomposition separates the finite-dimensional critical shoreline modes from the exponentially stable residual dynamics. State feedback is designed on the critical subspace and lifted to the full shoreline space.

    \item \textbf{Predictor-based delay compensation.}
    An auxiliary transport equation represents the constant input delay. A predictor is then constructed to remove the explicit delay from the nominal predictor dynamics.

    \item \textbf{Exponential stability of the full coupled system.}
    Exponential stability of the delayed closed-loop system, including the unactuated bathymetric subsystem, is established under a small-gain condition. The analysis combines predictor-based shoreline stabilization, bathymetric coercivity, and weighted input--output estimates for the hetero-dimensional interaction.

    \item \textbf{Numerical assessment.}
    Numerical simulations show that uncompensated delayed feedback can amplify both shoreline and bathymetric responses, whereas predictor compensation restores decay of the critical shoreline dynamics and prevents the resulting delay-induced amplification from propagating to the unactuated bathymetric subsystem.
\end{itemize}

The remainder of the paper is organized as follows.
Section~\ref{sec:model} introduces the coupled model.
The predictor-based feedback design and stability analysis are developed in Section~\ref{sec:control}, followed by the numerical study in Section~\ref{sec:numerics}.
The final section concludes the paper.

\section{Mathematical Model and Structural Properties}
\label{sec:model}

This section introduces the delayed hetero-dimensional shoreline--bathymetry model and identifies the coupling properties required for the subsequent stabilization analysis.

\subsection{Mathematical Model}
\label{subsec:mathematical-model}

Let
\[
I:=(0,\ell),
\qquad
\Omega:=I\times(0,d),
\qquad
\ell,d>0,
\]
and consider the hetero-dimensional system with constant input delay \(D>0\):
\begin{align}
\partial_t u(x,t)
&=
\epsilon_1 \partial_{xx}u(x,t)
+\alpha_1(x,t)u(x,t)\nonumber\\
&\quad
-\int_0^d
\lambda_1(x,y,t)\partial_xv(x,y,t)\,dy\nonumber\\
&\quad
+g(x)U(x,t-D),
\label{eq:u_pde}
\\
\partial_t v(x,y,t)
&=
\epsilon_2\partial_{xx}v(x,y,t)
+\bar\epsilon_2\partial_{yy}v(x,y,t)\nonumber\\
&\quad
+\alpha_2(x,y,t)v(x,y,t)
-\lambda_2(x,t)\partial_xu(x,t),
\label{eq:v_pde}
\end{align}
for \(x\in I\) in \eqref{eq:u_pde},
\((x,y)\in\Omega\) in \eqref{eq:v_pde}, and \(t>0\).

The diffusion coefficients satisfy
\[
\epsilon_1,\epsilon_2,\bar\epsilon_2>0,
\]
while
\[
\alpha_1\in L^\infty(I\times\mathbb R_+),
\qquad
\alpha_2\in L^\infty(\Omega\times\mathbb R_+),
\]
and
\[
\lambda_1\in L^\infty(\Omega\times\mathbb R_+),
\qquad
\lambda_2\in L^\infty(I\times\mathbb R_+)
\]
are real-valued. No sign condition is imposed on the coupling coefficients.

Set
\[
X:=L^2(I).
\]
The input \(U(\cdot,t)\in X\) acts on the shoreline subsystem through the spatial profile
\[
g\in L^\infty(I),
\qquad
\bigl|\{x\in I:\ g(x)\neq0\}\bigr|>0,
\]
and is subject to the constant input delay \(D\).

From a coastal-engineering perspective, the control term may be interpreted as an idealized sediment-management action, such as controlled beach nourishment \cite{WorkDean1995}.
The profile \(g\) describes the spatial distribution of the actuation, while \(U\) specifies the spatially distributed command.
The delay \(D\) represents an effective lag between command generation and the resulting morphodynamic action.
This interpretation is restricted to the control-oriented model: actuator saturation, unilateral sediment-placement constraints, and detailed sediment-transfer dynamics are not included.

The boundary conditions are
\begin{equation}
u(0,t)=u(\ell,t)=0,
\qquad t>0,
\label{eq:bc_u}
\end{equation}
and
\begin{equation}
\begin{cases}
\partial_xv(0,y,t)=\partial_xv(\ell,y,t)=0,
& (y,t)\in(0,d)\times\mathbb R_+,\\[1mm]
\partial_yv(x,0,t)=\partial_yv(x,d,t)=0,
& (x,t)\in I\times\mathbb R_+.
\end{cases}
\label{eq:bc_v}
\end{equation}
The shoreline Dirichlet conditions in \eqref{eq:bc_u} represent a coastal segment whose end points are pinned to the reference shoreline configuration, so that the modeled perturbation vanishes at the lateral boundaries.
The homogeneous Neumann conditions in \eqref{eq:bc_v} impose a zero normal bathymetric gradient at the boundary of the nearshore domain and, consequently, zero normal diffusive bathymetric flux through these artificial boundaries. These conditions prescribe the boundary behavior of the modeled perturbations and provide the boundary setting for the associated diffusion operators.

The initial data and input history are prescribed by
\begin{equation}
u(x,0)=u_0(x),
\qquad
v(x,y,0)=v_0(x,y),
\label{eq:init}
\end{equation}
and
\begin{equation}
U(\cdot,\tau)=U_0(\cdot,\tau)
\quad\text{in }X,
\quad\text{for a.e. }\tau\in(-D,0).
\label{eq:input_history}
\end{equation}

The state and energy spaces are
\[
\mathcal H:=L^2(I)\times L^2(\Omega),
\qquad
\mathcal V:=H_0^1(I)\times H^1(\Omega),
\]
equipped with
\[
\|(u,v)\|_{\mathcal H}^2
=
\|u\|_{L^2(I)}^2+\|v\|_{L^2(\Omega)}^2,
\]
and
\[
\|(u,v)\|_{\mathcal V}^2
=
\|u\|_{H_0^1(I)}^2+\|v\|_{H^1(\Omega)}^2.
\]
For
\[
(u_0,v_0)\in\mathcal H,
\qquad
U_0\in L^2(-D,0;X),
\]
the variational setting is based on the Gelfand triple
\[
\mathcal V\hookrightarrow\mathcal H\hookrightarrow\mathcal V'.
\]

\subsection{Structural Properties}
\label{subsec:structural-properties}

The shoreline--bathymetry interaction is represented by
\[
\begin{aligned}
C_{12}(t)&:H^1(\Omega)\to L^2(I),\\
(C_{12}(t)v)(x)
&=
\int_0^d
\lambda_1(x,y,t)\partial_xv(x,y)\,dy,
\end{aligned}
\]
and
\[
\begin{aligned}
C_{21}(t)&:H_0^1(I)\to L^2(\Omega),\\
(C_{21}(t)u)(x,y)
&=
\lambda_2(x,t)\partial_xu(x).
\end{aligned}
\]
Both coupling operators are well defined on the corresponding components of \(\mathcal V\) and satisfy bounds uniform in time, as stated next.

\begin{lemma}[Boundedness of the coupling operators]
\label{lem:bounded_coupling}
For almost every \(t\ge0\),
\[
\|C_{12}(t)v\|_{L^2(I)}
\le
M_{12}\|v\|_{H^1(\Omega)},
\]
and
\[
\|C_{21}(t)u\|_{L^2(\Omega)}
\le
M_{21}\|u\|_{H_0^1(I)},
\]
where
\[
M_{12}
=
\sqrt d\,
\|\lambda_1\|_{L^\infty(\Omega\times\mathbb R_+)},
\qquad
M_{21}
=
\sqrt d\,
\|\lambda_2\|_{L^\infty(I\times\mathbb R_+)}.
\]
\end{lemma}

\begin{proof}
Let \(v\in H^1(\Omega)\). By the Cauchy--Schwarz inequality in the \(y\)-variable,
\[
\begin{aligned}
|(C_{12}(t)v)(x)|^2
&\le
\left(
\int_0^d|\lambda_1(x,y,t)|^2\,dy
\right)\\
&\quad\times
\left(
\int_0^d|\partial_xv(x,y)|^2\,dy
\right)\\
&\le
d\|\lambda_1\|_{L^\infty}^2
\int_0^d|\partial_xv(x,y)|^2\,dy.
\end{aligned}
\]
Integration over \(I\) gives
\[
\begin{aligned}
\|C_{12}(t)v\|_{L^2(I)}^2
&\le
d\|\lambda_1\|_{L^\infty}^2
\|\partial_xv\|_{L^2(\Omega)}^2\\
&\le
M_{12}^2\|v\|_{H^1(\Omega)}^2.
\end{aligned}
\]

Similarly, for \(u\in H_0^1(I)\),
\[
\begin{aligned}
\|C_{21}(t)u\|_{L^2(\Omega)}^2
&=
\int_0^\ell\int_0^d
|\lambda_2(x,t)|^2
|\partial_xu(x)|^2\,dy\,dx\\
&\le
d\|\lambda_2\|_{L^\infty}^2
\|\partial_xu\|_{L^2(I)}^2\\
&\le
M_{21}^2\|u\|_{H_0^1(I)}^2.
\end{aligned}
\]
Taking square roots completes the proof.
\end{proof}

It follows from Lemma~\ref{lem:bounded_coupling} that
\[
\mathcal C(t)(u,v)
=
\bigl(-C_{12}(t)v,-C_{21}(t)u\bigr)
\]
defines a uniformly bounded operator from \(\mathcal V\) into \(\mathcal H\) for almost every \(t\geq0\).
This boundedness property provides no information on the sign of the associated energy contribution. In particular, it implies neither dissipativity nor energy cancellation, as established by the following proposition.

\begin{proposition}[Sign-indefinite hetero-dimensional interaction]
\label{prop:sign-indefinite}
Under the standing assumptions, the hetero-dimensional coupling need not be sign-definite. More precisely, there exists an admissible choice of coupling coefficients for which
\[
\left\langle
\mathcal C(t)(u,v),(u,v)
\right\rangle_{\mathcal H}
\]
takes both positive and negative values as \((u,v)\) varies in \(\mathcal V\). Consequently, neither dissipativity nor cancellation of the coupling contribution in the energy identity can be inferred from the standing assumptions alone.
\end{proposition}

\begin{proof}
For \((u,v)\in\mathcal V\),
\[
\left\langle
\mathcal C(t)(u,v),(u,v)
\right\rangle_{\mathcal H}
=
-\left\langle C_{12}(t)v,u\right\rangle_{L^2(I)}
-\left\langle C_{21}(t)u,v\right\rangle_{L^2(\Omega)}.
\]
To exhibit the sign-indefinite character, consider the admissible coefficients
\[
\lambda_1\equiv0,
\qquad
\lambda_2\equiv1,
\]
and choose any nonzero \(u\in C_c^\infty(I)\).

Define the \(y\)-independent functions
\[
v_\pm(x,y)=\pm u_x(x).
\]
Since \(u_x\in C_c^\infty(I)\), one has \(v_\pm\in H^1(\Omega)\). Moreover,
\[
C_{12}(t)v_\pm=0,
\qquad
C_{21}(t)u=u_x.
\]
Hence
\[
\begin{aligned}
\left\langle
\mathcal C(t)(u,v_\pm),(u,v_\pm)
\right\rangle_{\mathcal H}
&=
-\int_0^\ell\int_0^d
u_x(x)\,v_\pm(x,y)\,dy\,dx\\
&=
\mp d\,\|u_x\|_{L^2(I)}^2.
\end{aligned}
\]
Since \(u\neq0\), these values are nonzero and have opposite signs. Thus, for the same admissible coefficient configuration, the coupling contribution takes both signs as the state varies.
\end{proof}

The control objective is therefore to compensate the delayed shoreline actuation and establish exponential stability of the full coupled system without direct actuation of the bathymetric component, despite the sign-indefinite hetero-dimensional interaction. The predictor-based stabilization design is developed in the next section.

\section{Predictor-Based Spectral Stabilization}
\label{sec:control}

This section develops the stabilization design for the delayed coupled system. The construction combines spectral feedback on the finite-dimensional critical shoreline subspace with predictor compensation of the actuator delay, followed by a stability analysis of the full hetero-dimensional dynamics.

\subsection{Predictor-Based Spectral Feedback Design}
\label{subsec:feedback-design}

Since only the shoreline subsystem is directly actuated, the feedback is constructed from a nominal shoreline operator and then coupled with a predictor representation of the delayed input. To this end, let
\[
\bar{\alpha}_1\in L^\infty(I;\mathbb R)
\]
be a prescribed time-independent reference coefficient used for the spectral feedback design. It may be chosen as a representative value of the time-dependent coefficient \(\alpha_1(\cdot,t)\); no particular averaging rule is required by the analysis. The difference \(\alpha_1(\cdot,t)-\bar{\alpha}_1\) is retained explicitly in the closed-loop analysis.

Define the nominal shoreline operator
\[
A_n\phi
=
\epsilon_1\phi_{xx}
+
\bar{\alpha}_1(x)\phi,
\qquad
D(A_n)=H^2(I)\cap H_0^1(I),
\]
on \(X=L^2(I)\), together with the bounded actuation operator
\[
(B\psi)(x)=g(x)\psi(x),
\qquad
B\in\mathcal L(X).
\]
The shoreline equation can then be written as
\begin{equation}
\dot u(t)
=
A_nu(t)+BU(t-D)+\mathfrak d(t),
\label{eq:shoreline-residual-form}
\end{equation}
where
\[
\mathfrak d(t)
=
\bigl(\alpha_1(\cdot,t)-\bar{\alpha}_1\bigr)u(t)
-
C_{12}(t)v(t).
\]
Thus, \(A_n\) serves as the nominal generator underlying the spectral decomposition, feedback design, and predictor construction, while the time dependence of \(\alpha_1\) and the shoreline--bathymetry coupling remain present in the full dynamics through \(\mathfrak d\).

Write
\[
A_n=A_D+M_{\bar\alpha_1},
\qquad
\begin{cases}
A_D\phi=\epsilon_1\phi_{xx},\\
D(A_D)=H^2(I)\cap H_0^1(I),
\end{cases}
\]
where \(M_{\bar\alpha_1}\) denotes multiplication by \(\bar\alpha_1\).
The Dirichlet diffusion operator \(A_D\) is self-adjoint on \(L^2(I)\). Since \(\bar\alpha_1\) is real-valued and bounded, \(M_{\bar\alpha_1}\) is bounded and self-adjoint.
Consequently, \(A_n\), with \(D(A_n)=D(A_D)\), is self-adjoint.

Moreover, standard bounded-perturbation results for analytic semigroups \cite{Pazy1983} imply that \(A_n\) generates an analytic semigroup on \(X\).
Finally, the compact embedding
\[
D(A_n)=H^2(I)\cap H_0^1(I)
\hookrightarrow L^2(I)
\]
implies that \(A_n\) has compact resolvent.

Let \(\rho>0\) be fixed, and let \(P_{\rm cr}\) denote the orthogonal spectral projection of \(A_n\) associated with the eigenvalues contained in \((-\rho,\infty)\). Set
\[
P_{\rm st}=I-P_{\rm cr},
\qquad
X_{\rm cr}=P_{\rm cr}X,
\qquad
X_{\rm st}=P_{\rm st}X,
\]
so that
\[
X=X_{\rm cr}\oplus X_{\rm st}.
\]
Hereafter, the subscripts ``cr'' and ``st'' refer to the critical and stable residual components, respectively.
Since \(A_n\) is self-adjoint, bounded from above, and has compact resolvent, its spectrum consists of real isolated eigenvalues of finite multiplicity that can accumulate only at \(-\infty\). Hence \(X_{\rm cr}\) is finite-dimensional.
Moreover, with
\[
A_{\rm st}:=A_n|_{X_{\rm st}},
\]
one has
\[
\sigma(A_{\rm st})\subset(-\infty,-\rho],
\]
and therefore
\begin{equation}
\|e^{A_{\rm st}t}\|_{\mathcal L(X_{\rm st})}
\le e^{-\rho t},
\qquad t\ge0.
\label{eq:stable-residual-decay}
\end{equation}
Finally, define
\[
A_{\rm cr}:=A_n|_{X_{\rm cr}},
\qquad
B_{\rm cr}:=P_{\rm cr}B\in\mathcal L(X,X_{\rm cr}).
\]

The nominal feedback design is therefore based on the finite-dimensional critical pair \((A_{\rm cr},B_{\rm cr})\), whose controllability is established next.

\begin{proposition}[Controllability of the critical subsystem]
\label{prop:critical-controllability}
Assume that
\[
g\in L^\infty(I;\mathbb R),
\qquad
\bigl|\{x\in I:\ g(x)\neq0\}\bigr|>0.
\]
Then the pair \((A_{\rm cr},B_{\rm cr})\) is controllable.
Consequently, there exists
\[
K_{\rm cr}\in\mathcal L(X_{\rm cr},X)
\]
such that \(A_{\rm cr}+B_{\rm cr}K_{\rm cr}\) is Hurwitz.
In particular, there exist \(M_{\rm cr}\ge1\) and \(\omega_{\rm cr}>0\) such that
\begin{equation}
\left\|
e^{(A_{\rm cr}+B_{\rm cr}K_{\rm cr})t}
\right\|_{\mathcal L(X_{\rm cr})}
\le
M_{\rm cr}e^{-\omega_{\rm cr}t},
\qquad t\ge0.
\label{eq:critical-exponential-decay}
\end{equation}
\end{proposition}

\begin{proof}
Since \(A_n\) is self-adjoint, \(A_{\rm cr}^*=A_{\rm cr}\).
Moreover, the real-valuedness of \(g\) implies \(B^*=B\).
Hence, for \(\phi\in X_{\rm cr}\),
\[
B_{\rm cr}^*\phi
=
B^*P_{\rm cr}\phi
=
g\phi.
\]
By the finite-dimensional Popov--Belevitch--Hautus (PBH) controllability criterion, it is enough to show that
\begin{equation}
\ker(\lambda I-A_{\rm cr}^*)
\cap
\ker B_{\rm cr}^*
=
\{0\},
\qquad
\lambda\in\sigma(A_{\rm cr}).
\label{eq:PBH-critical}
\end{equation}

Let \(\phi\) belong to the intersection in \eqref{eq:PBH-critical}. Then
\[
A_n\phi=\lambda\phi,
\qquad
g\phi=0
\quad\text{a.e. on }I.
\]
Thus \(\phi=0\) a.e. on
\[
E_g:=\{x\in I:\ g(x)\neq0\},
\qquad
|E_g|>0.
\]
Since
\[
\phi\in H^2(I)\cap H_0^1(I)
\hookrightarrow C^1(\overline I),
\]
the set of pointwise zeros of \(\phi\) has positive measure and hence has an accumulation point \(x_\ast\in I\).
Choosing distinct zeros \(x_n\to x_\ast\), continuity gives \(\phi(x_\ast)=0\), while
\[
\phi'(x_\ast)
=
\lim_{n\to\infty}
\frac{\phi(x_n)-\phi(x_\ast)}
{x_n-x_\ast}
=
0.
\]
Using the definition of \(A_n\), the eigenvalue equation becomes
\[
\epsilon_1\phi''+\bar\alpha_1\phi=\lambda\phi,
\]
or equivalently,
\[
\phi''
=
\frac{\lambda-\bar\alpha_1}{\epsilon_1}\phi
\quad\text{a.e. on }I,
\]
where the coefficient belongs to \(L^\infty(I)\).
Writing this equation as a first-order system for \((\phi,\phi')\), uniqueness of the corresponding linear initial-value problem with
\[
\phi(x_\ast)=\phi'(x_\ast)=0
\]
yields \(\phi\equiv0\). Hence \eqref{eq:PBH-critical} holds, and the pair \((A_{\rm cr},B_{\rm cr})\) is controllable. The existence of \(K_{\rm cr}\) and the exponential estimate then follow from finite-dimensional pole assignment.
\end{proof}

The critical feedback must next be lifted to the full shoreline space.
Since multiplication by \(g\) need not preserve \(X_{\rm cr}\), the feedback action may generate a component in \(X_{\rm st}\). This residual component must be accounted for when passing from critical-mode stabilization to the full shoreline dynamics.

\begin{theorem}[Finite-to-infinite lifting of critical-mode stabilization]
\label{thm:finite-infinite-lifting}
Let \(K_{\rm cr}\) satisfy \eqref{eq:critical-exponential-decay}.
Define
\begin{equation}
A_F
:=
A_n+BK_{\rm cr}P_{\rm cr},
\qquad
D(A_F)=D(A_n).
\label{eq:AF-lifting}
\end{equation}
Then \(A_F\) generates an analytic exponentially stable \(C_0\)-semigroup on \(X\). More precisely, for every
\[
0<\omega<\min\{\rho,\omega_{\rm cr}\},
\]
there exists \(M_F\ge1\) such that
\begin{equation}
\|e^{A_Ft}\|_{\mathcal L(X)}
\le
M_Fe^{-\omega t},
\qquad t\ge0.
\label{eq:AF-exp-lifting}
\end{equation}
\end{theorem}

\begin{proof}
Since \(BK_{\rm cr}P_{\rm cr}\in\mathcal L(X)\), the standard bounded-perturbation theorem for analytic semigroups \cite{Pazy1983} implies that \(A_F\), with \(D(A_F)=D(A_n)\), generates an analytic semigroup on \(X\).

Relative to \(X=X_{\rm cr}\oplus X_{\rm st}\),
\[
A_F
=
\begin{pmatrix}
A_{\rm cr}+B_{\rm cr}K_{\rm cr} & 0\\[1mm]
L_{\rm sc} & A_{\rm st}
\end{pmatrix},
\qquad
L_{\rm sc}:=P_{\rm st}BK_{\rm cr},
\]
with domain \(X_{\rm cr}\oplus D(A_{\rm st})\).
Indeed, the spectral projections commute with \(A_n\), while \(P_{\rm cr}P_{\rm st}=0\). Hence, for \(q=q_{\rm cr}+q_{\rm st}\),
\[
\dot q_{\rm cr}
=
(A_{\rm cr}+B_{\rm cr}K_{\rm cr})q_{\rm cr},
\qquad
\dot q_{\rm st}
=
A_{\rm st}q_{\rm st}+L_{\rm sc}q_{\rm cr}.
\]

Using \eqref{eq:stable-residual-decay}, \eqref{eq:critical-exponential-decay}, and variation of constants gives
\[
\begin{aligned}
\|q_{\rm st}(t)\|_X
&\le
e^{-\rho t}\|q_{\rm st}(0)\|_X\\
&\quad+
M_{\rm cr}\|L_{\rm sc}\|
\int_0^t
e^{-\rho(t-s)}e^{-\omega_{\rm cr}s}\,ds\,
\|q_{\rm cr}(0)\|_X.
\end{aligned}
\]
Setting \(m=\min\{\rho,\omega_{\rm cr}\}\), one has
\[
\int_0^t
e^{-\rho(t-s)}e^{-\omega_{\rm cr}s}\,ds
\le
t e^{-mt}.
\]
For every \(0<\omega<m\),
\[
t e^{-mt}
=
t e^{-(m-\omega)t}e^{-\omega t}
\le
\frac{1}{e(m-\omega)}e^{-\omega t}.
\]
Together with the critical-mode estimate and the orthogonal decomposition of \(X\), this yields \eqref{eq:AF-exp-lifting}.
\end{proof}

Fix
\[
0<\omega_F<\min\{\rho,\omega_{\rm cr}\},
\]
and let \(M_F\ge1\) denote the corresponding constant in \eqref{eq:AF-exp-lifting}.

The actuator delay is incorporated into the feedback design through the transport state
\[
z(\theta,t)
=
U(t-D+\theta),
\qquad
\theta\in(0,D),
\]
which satisfies
\begin{equation}
\begin{cases}
z_t(\theta,t)=z_\theta(\theta,t),
&(\theta,t)\in(0,D)\times\mathbb R_+,\\
z(D,t)=U(t),\\
z(0,t)=U(t-D).
\end{cases}
\label{eq:transport}
\end{equation}
Equation \eqref{eq:shoreline-residual-form} becomes
\begin{equation}
\dot u
=
A_nu+Bz(0,t)+\mathfrak d.
\label{eq:shoreline_transport}
\end{equation}

Define the predictor state
\begin{equation}
Y(t)
=
e^{A_nD}u(t)
+
\int_0^D
e^{A_n(D-\theta)}
Bz(\theta,t)\,d\theta.
\label{eq:predictor}
\end{equation}
Its initial value is determined by the plant state and input history as
\begin{equation}
Y(0)
=
e^{A_nD}u_0
+
\int_0^D
e^{A_n(D-\theta)}
B\,U_0(-D+\theta)\,d\theta.
\label{eq:predictor-initial}
\end{equation}
The next proposition shows that the predictor transformation removes the input delay from the nominal shoreline dynamics while preserving the residual term \(\mathfrak d\).

\begin{proposition}[Predictor dynamics]
\label{prop:predictor-dynamics}
Let \(u\) and \(z\) satisfy \eqref{eq:transport}--\eqref{eq:shoreline_transport}.
Then the predictor state \eqref{eq:predictor} satisfies
\begin{equation}
\dot Y
=
A_nY+BU(t)+e^{A_nD}\mathfrak d(t).
\label{eq:predictor-dynamics}
\end{equation}
The identity holds classically for sufficiently regular solutions and extends to energy solutions in the mild sense.
\end{proposition}

\begin{proof}
For sufficiently regular solutions, differentiation of \eqref{eq:predictor} and \(z_t=z_\theta\) gives
\[
\dot Y
=
e^{A_nD}\dot u
+
\int_0^D
e^{A_n(D-\theta)}Bz_\theta(\theta,t)\,d\theta.
\]
Integration by parts and \(z(D,t)=U(t)\) yield
\[
\begin{aligned}
\int_0^D
e^{A_n(D-\theta)}Bz_\theta\,d\theta
&=
BU(t)-e^{A_nD}Bz(0,t)\\
&\quad+
A_n\int_0^D
e^{A_n(D-\theta)}
Bz(\theta,t)\,d\theta.
\end{aligned}
\]
Substituting \eqref{eq:shoreline_transport} gives
\[
\begin{aligned}
\dot Y
&=
e^{A_nD}A_nu
+e^{A_nD}Bz(0,t)
+e^{A_nD}\mathfrak d(t)\\
&\quad+
BU(t)-e^{A_nD}Bz(0,t)\\
&\quad+
A_n\int_0^D
e^{A_n(D-\theta)}
Bz(\theta,t)\,d\theta.
\end{aligned}
\]
The terms containing \(Bz(0,t)\) cancel. Using
\[
e^{A_nD}A_nu=A_ne^{A_nD}u
\]
and the definition \eqref{eq:predictor} gives
\[
\dot Y
=
A_nY+BU(t)+e^{A_nD}\mathfrak d(t),
\]
which is \eqref{eq:predictor-dynamics}.

By approximation of the initial data and input history, the identity extends to energy solutions through the variation-of-constants formula
\[
\begin{aligned}
Y(t)
&=
e^{A_nt}Y(0)\\
&\quad+
\int_0^t e^{A_n(t-s)}
\bigl[
BU(s)+e^{A_nD}\mathfrak d(s)
\bigr]\,ds.
\end{aligned}
\]
\end{proof}

The explicit input delay is thus removed from the predictor dynamics.
With the bounded lifted feedback operator
\[
F:=K_{\rm cr}P_{\rm cr}\in\mathcal L(X),
\]
choose
\begin{equation}
U(t)=FY(t)=K_{\rm cr}P_{\rm cr}Y(t).
\label{eq:predictor-feedback}
\end{equation}
Substitution into \eqref{eq:predictor-dynamics}, together with the definition \eqref{eq:AF-lifting}, gives
\begin{equation}
\dot Y
=
A_FY+e^{A_nD}\mathfrak d,
\label{eq:predictor-closed-loop}
\end{equation}
where, by Proposition~\ref{prop:critical-controllability} and Theorem~\ref{thm:finite-infinite-lifting}, \(A_F\) generates an exponentially stable analytic semigroup satisfying \eqref{eq:AF-exp-lifting}.
It remains to establish that the resulting predictor-based interconnection defines a globally well-posed closed-loop evolution.

\begin{theorem}[Global well-posedness of the closed-loop system]
\label{thm:global-wellposedness}
Under the coefficient, actuator, and coupling conditions of Section~\ref{sec:model}, let \(F=K_{\rm cr}P_{\rm cr}\) be defined by \eqref{eq:predictor-feedback}. For every
\[
(u_0,v_0)\in\mathcal H,
\qquad
U_0\in L^2(-D,0;X),
\]
the predictor feedback generates a unique global energy solution such that, for every \(T>0\),
\[
(u,v)
\in
L^2(0,T;\mathcal V)
\cap
H^1(0,T;\mathcal V')
\cap
C([0,T];\mathcal H),
\]
\[
U\in L^2(-D,T;X),
\;
z\in C([0,T];\mathcal Z),
\;
\mathcal Z:=L^2((0,D);X),
\]
where \(z(\theta,t)=U(t-D+\theta)\).
Moreover, there exists \(C_T>0\) such that
\begin{align}
&\sup_{0\le t\le T}
\|(u(t),v(t))\|_{\mathcal H}^2
+
\int_0^T
\|(u(t),v(t))\|_{\mathcal V}^2\,dt
\nonumber\\
&\quad+
\|U\|_{L^2(0,T;X)}^2
+
\sup_{0\le t\le T}\|z(t)\|_{\mathcal Z}^2
\nonumber\\
&\le
C_T
\left(
\|(u_0,v_0)\|_{\mathcal H}^2
+
\|U_0\|_{L^2(-D,0;X)}^2
\right).
\label{eq:closed-loop-finite-time-estimate}
\end{align}
\end{theorem}

\begin{proof}
For \(\xi=(u,v)\) and \(\eta=(\phi,\psi)\) in \(\mathcal V\), define
\begin{align}
a_t(\xi,\eta)
&=
\epsilon_1\langle u_x,\phi_x\rangle_{L^2(I)}
+
\epsilon_2\langle v_x,\psi_x\rangle_{L^2(\Omega)}
\nonumber\\
&\quad+
\bar\epsilon_2\langle v_y,\psi_y\rangle_{L^2(\Omega)}
-
\langle\alpha_1u,\phi\rangle_{L^2(I)}
\nonumber\\
&\quad-
\langle\alpha_2v,\psi\rangle_{L^2(\Omega)}
+
\langle C_{12}(t)v,\phi\rangle_{L^2(I)}
\nonumber\\
&\quad+
\langle C_{21}(t)u,\psi\rangle_{L^2(\Omega)}.
\label{eq:closed-loop-variational-form}
\end{align}
Lemma~\ref{lem:bounded_coupling}, boundedness of the reaction coefficients, Young's inequality, and Poincar\'e's inequality for the shoreline component imply
\[
|a_t(\xi,\eta)|
\le
C_a\|\xi\|_{\mathcal V}\|\eta\|_{\mathcal V}
\]
and, for suitable \(c_{\mathcal V}>0\), \(c_{\mathcal H}\ge0\),
\begin{equation}
a_t(\xi,\xi)
+
c_{\mathcal H}\|\xi\|_{\mathcal H}^2
\ge
c_{\mathcal V}\|\xi\|_{\mathcal V}^2.
\label{eq:garding-closed-loop}
\end{equation}
Thus, for every prescribed \(\widehat U\in L^2(-D,T;X)\), standard variational theory for nonautonomous parabolic equations \cite{LionsMagenes1972,DautrayLions1992} provides a unique energy solution and
\begin{equation}
\begin{aligned}
\sup_{0\le t\le T}\|\xi(t)\|_{\mathcal H}^2
&+
\int_0^T\|\xi(t)\|_{\mathcal V}^2\,dt\\
&\le
C_T
\left(
\|\xi_0\|_{\mathcal H}^2
+
\|\widehat U\|_{L^2(-D,T-D;X)}^2
\right).
\end{aligned}
\label{eq:prescribed-input-estimate}
\end{equation}

Using \eqref{eq:predictor} and the transport representation, the feedback can equivalently be written as the causal Volterra equation
\begin{equation}
U(t)
=
Fe^{A_nD}u(t)
+
\int_{t-D}^{t}
Fe^{A_n(t-r)}BU(r)\,dr.
\label{eq:feedback-volterra-equation}
\end{equation}
Set
\[
K_D(s)
=
\mathbf 1_{[0,D]}(s)Fe^{A_ns}B,
\qquad s\ge0.
\]
Then \eqref{eq:feedback-volterra-equation} can be written, for \(t\ge0\), as
\[
U(t)
=
Fe^{A_nD}u(t)
+
\int_0^tK_D(t-r)U(r)\,dr
+
h_0(t),
\]
where
\[
h_0(t)
=
\int_{-D}^{0}
K_D(t-r)U_0(r)\,dr.
\]
Since
\[
K_D\in L^1(0,T;\mathcal L(X)),
\]
Young's convolution inequality gives
\[
\|h_0\|_{L^2(0,T;X)}
\le
\|K_D\|_{L^1(0,T;\mathcal L(X))}
\|U_0\|_{L^2(-D,0;X)}.
\]
The kernel \(K_D\) is bounded on finite intervals.
The Volterra equation is therefore uniquely solvable on sufficiently short intervals by contraction, and successive continuation covers \([0,T]\). The corresponding finite-time estimate gives
\begin{equation}
\|U\|_{L^2(0,T;X)}
\le
C_T
\left(
\|u\|_{L^2(0,T;X)}
+
\|U_0\|_{L^2(-D,0;X)}
\right).
\label{eq:feedback-L2-estimate}
\end{equation}

The coupled plant--Volterra system is causal.
On the first interval \([0,D]\), the delayed plant input is prescribed entirely by the history \(U_0\).
Hence the variational problem uniquely determines \((u,v)\) on \([0,D]\), after which the Volterra equation uniquely determines \(U\) on the same interval. Inductively, assume that \(U\) has been constructed on \([-D,kD]\). For \(t\in[kD,(k+1)D]\), one has
\[
t-D\in[(k-1)D,kD],
\]
so that the delayed plant input \(U(t-D)\) is already known. The variational problem therefore uniquely determines \((u,v)\) on \([kD,(k+1)D]\), and the Volterra equation subsequently determines \(U\) there.
The estimates \eqref{eq:prescribed-input-estimate} and \eqref{eq:feedback-L2-estimate} propagate through this induction.

Thus, for every finite \(T>0\), the closed-loop system admits a unique energy solution on \([0,T]\), with a bound whose constant depends on \(T\), the system coefficients, the delay, and the fixed feedback operator, but not on the initial data.

Finally,
\[
\|z(t)\|_{\mathcal Z}^2
=
\int_{t-D}^{t}\|U(r)\|_X^2\,dr,
\]
with \(U=U_0\) on \([-D,0]\). Hence
\[
\sup_{0\le t\le T}\|z(t)\|_{\mathcal Z}^2
\le
\|U_0\|_{L^2(-D,0;X)}^2
+
\|U\|_{L^2(0,T;X)}^2,
\]
and continuity of translations in \(L^2\) gives \(z\in C([0,T];\mathcal Z)\). Combining the preceding estimates proves \eqref{eq:closed-loop-finite-time-estimate}.
\end{proof}

\subsection{Closed-Loop Stability Analysis}
\label{subsec:closed-loop-stability}

The remaining closed-loop interaction is carried by
\[
\mathfrak d(t)
=
\widetilde\alpha_1(\cdot,t)u(t)
-
C_{12}(t)v(t),
\]
where
\[
\widetilde\alpha_1(x,t)
=
\alpha_1(x,t)-\bar\alpha_1(x).
\]
By Lemma~\ref{lem:bounded_coupling},
\[
\|\mathfrak{d}(t)\|_X
\le
m_\alpha\|u(t)\|_X
+
M_{12}\|v(t)\|_{H^1(\Omega)},
\]
where
\[
m_\alpha
:=
\|\widetilde\alpha_1\|_{L^\infty(I\times\mathbb R_+)}.
\]
Theorem~\ref{thm:global-wellposedness} therefore implies
\[
\mathfrak d\in L^2(0,T;X)
\qquad
\text{for every }T>0.
\]

For a Hilbert space \(E\) and \(\theta>0\), define
\[
\|f\|_{L^2_\theta(0,T;E)}
:=
\left(
\int_0^T e^{2\theta t}\|f(t)\|_E^2\,dt
\right)^{1/2}.
\]
The two weighted mappings \(\mathfrak d\mapsto u_x\) and \(u_x\mapsto\mathfrak d\) are estimated separately below.

\begin{proposition}[Weighted predictor and shoreline-gradient estimate]
\label{prop:weighted-shoreline-gain}
Let
\[
0<\theta<\omega_F,
\]
and let
\[
\dot Y=A_FY+e^{A_nD}\mathfrak d,
\qquad
Y(0)=Y_0,
\]
where
\[
\|e^{A_Ft}\|_{\mathcal L(X)}
\le
M_Fe^{-\omega_Ft}.
\]
Then, for every \(T>0\),
\begin{equation}
\label{eq:predictor-weighted-estimate}
\|Y\|_{L^2_\theta(0,T;X)}
\le
C_{Y0}(\theta)\|Y_0\|_X
+
G_{Y\mathfrak d}(\theta)
\|\mathfrak d\|_{L^2_\theta(0,T;X)},
\end{equation}
where
\[
C_{Y0}(\theta)
=
\frac{M_F}{\sqrt{2(\omega_F-\theta)}},
\qquad
G_{Y\mathfrak d}(\theta)
=
\frac{M_F\|e^{A_nD}\|_{\mathcal L(X)}}
{\omega_F-\theta}.
\]
Moreover, there exist constants \(C_{Y,1}(\theta),G_{Y\mathfrak d}^{(1)}(\theta)>0\), independent of \(T\), such that
\begin{equation}
\label{eq:weighted-Y-H1}
\|Y\|_{L^2_\theta(0,T;H_0^1(I))}
\le
C_{Y,1}(\theta)\|Y_0\|_X
+
G_{Y\mathfrak d}^{(1)}(\theta)
\|\mathfrak d\|_{L^2_\theta(0,T;X)}.
\end{equation}
For every \(T>D\),
\begin{align}
\|u_x\|_{L^2_\theta(D,T;L^2(I))}
&\le
e^{\theta D}C_{Y,1}(\theta)\|Y_0\|_X
\nonumber\\
&\quad+
G_{u\mathfrak d}(\theta)
\|\mathfrak d\|_{L^2_\theta(0,T;X)},
\label{eq:weighted-ux-gain}
\end{align}
where
\begin{equation}
\label{eq:Gud}
G_{u\mathfrak d}(\theta)
=
e^{\theta D}G_{Y\mathfrak d}^{(1)}(\theta)+G_D(\theta),
\end{equation}
with
\begin{equation}
\label{eq:GD}
G_D(\theta)
=
C_n\int_0^D
r^{-1/2}e^{(\omega_n+\theta)r}\,dr
<\infty.
\end{equation}
\end{proposition}

\begin{proof}
The variation-of-constants formula and the exponential stability of \(A_F\) give
\[
\begin{aligned}
e^{\theta t}\|Y(t)\|_X
&\le
M_Fe^{-(\omega_F-\theta)t}\|Y_0\|_X\\
&+
M_F\|e^{A_nD}\|_{\mathcal L(X)}
\int_0^t
e^{-(\omega_F-\theta)(t-s)}
e^{\theta s}\|\mathfrak d(s)\|_X\,ds.
\end{aligned}
\]
Young's convolution inequality yields \eqref{eq:predictor-weighted-estimate} with the stated constants.

Since
\[
A_F=A_n+BF,
\qquad
BF\in\mathcal L(X),
\]
the bounded feedback term is of lower order with respect to the Dirichlet diffusion operator. Testing the predictor equation against \(Y\), multiplying by \(e^{2\theta t}\), and using Young's inequality gives
\[
\begin{aligned}
\epsilon_1
\|Y_x\|_{L^2_\theta(0,T;L^2(I))}^2
&\le
\frac12\|Y_0\|_X^2
+
(c_F+\theta)
\|Y\|_{L^2_\theta(0,T;X)}^2\\
&\quad+
c_{\mathfrak d}
\|\mathfrak d\|_{L^2_\theta(0,T;X)}^2,
\end{aligned}
\]
for constants \(c_F,c_{\mathfrak d}>0\) independent of \(T\).
Combining this inequality with \eqref{eq:predictor-weighted-estimate} and Poincar\'e's inequality proves \eqref{eq:weighted-Y-H1}.

For \(t\ge D\), the predictor identity gives
\begin{equation}
\label{eq:u-predictor-weighted}
u(t)
=
Y(t-D)
+
\int_{t-D}^{t}
e^{A_n(t-s)}\mathfrak d(s)\,ds.
\end{equation}
Analyticity of the semigroup generated by \(A_n\) yields constants \(C_n>0\) and \(\omega_n\in\mathbb R\) such that
\begin{equation}
\label{eq:Sn-smoothing}
\|e^{A_nr}\phi\|_{H_0^1(I)}
\le
C_nr^{-1/2}e^{\omega_nr}\|\phi\|_X,
\qquad r>0.
\end{equation}
The weighted convolution kernel
\[
C_nr^{-1/2}e^{(\omega_n+\theta)r}
\mathbf 1_{(0,D)}(r)
\]
belongs to \(L^1(0,D)\). Hence Young's convolution inequality, together with
\[
\|Y(\cdot-D)\|_{L^2_\theta(D,T;H_0^1(I))}
=
e^{\theta D}
\|Y\|_{L^2_\theta(0,T-D;H_0^1(I))},
\]
gives
\[
\begin{aligned}
\|u_x\|_{L^2_\theta(D,T;L^2(I))}
&\le
e^{\theta D}C_{Y,1}(\theta)\|Y_0\|_X\\
&\quad+
\left(
e^{\theta D}G_{Y\mathfrak d}^{(1)}(\theta)
+
G_D(\theta)
\right)
\|\mathfrak d\|_{L^2_\theta(0,T;X)}.
\end{aligned}
\]
This is \eqref{eq:weighted-ux-gain}.
\end{proof}

For the bathymetric subsystem, define
\begin{equation}
\label{eq:bathymetric-form}
\begin{aligned}
\mathfrak a_v(t;\phi)
&=
\epsilon_2\|\phi_x\|_{L^2(\Omega)}^2
+
\bar\epsilon_2\|\phi_y\|_{L^2(\Omega)}^2\\
&\quad
-
\int_\Omega
\alpha_2(x,y,t)|\phi(x,y)|^2\,dx\,dy.
\end{aligned}
\end{equation}
This form isolates the intrinsic bathymetric contribution to the energy balance. Its coercivity and the resulting weighted estimate are established next.

\begin{proposition}[Bathymetric coercivity and weighted gain]
\label{prop:bathymetric-weighted-gain}
Assume that there exists \(\mu_v>0\) such that, for almost every \(t\ge0\),
\begin{equation}
\label{eq:bathymetric-coercivity}
\mathfrak a_v(t;\phi)
\ge
\mu_v\|\phi\|_{H^1(\Omega)}^2,
\qquad
\phi\in H^1(\Omega).
\end{equation}
Then every energy solution satisfies, for almost every \(t>0\),
\begin{equation}
\label{eq:bathymetric-energy-gain}
\frac12\frac{d}{dt}\|v(t)\|_{L^2(\Omega)}^2
+
\frac{\mu_v}{2}\|v(t)\|_{H^1(\Omega)}^2
\le
b_v\|u_x(t)\|_{L^2(I)}^2,
\end{equation}
where
\begin{equation}
\label{eq:bv-definition}
b_v
=
\frac{dL_2^2}{2\mu_v},
\qquad
L_2
=
\|\lambda_2\|_{L^\infty(I\times\mathbb R_+)}.
\end{equation}
Moreover, for every
\[
0<\theta<\frac{\mu_v}{2}
\]
and every \(T>0\),
\begin{align}
\|v\|_{L^2_\theta(0,T;H^1(\Omega))}
&\le
C_{v0}(\theta)\|v_0\|_{L^2(\Omega)}
\nonumber\\
&\quad+
G_{vu}(\theta)
\|u_x\|_{L^2_\theta(0,T;L^2(I))},
\label{eq:weighted-bathymetric-gain}
\end{align}
where
\begin{equation}
\label{eq:bathymetric-weighted-constants}
C_{v0}(\theta)
=
\frac{1}{\sqrt{\mu_v-2\theta}},
\qquad
G_{vu}(\theta)
=
\sqrt{\frac{2b_v}{\mu_v-2\theta}}.
\end{equation}
\end{proposition}

\begin{proof}
Testing the bathymetric equation with \(v(t)\) gives
\[
\begin{aligned}
\frac12\frac{d}{dt}\|v(t)\|_{L^2(\Omega)}^2
&+
\mathfrak a_v(t;v(t))\\
&=
-\int_\Omega\lambda_2(x,t)u_x(x,t)v(x,y,t)\,dx\,dy.
\end{aligned}
\]
Since \(u_x\) is independent of \(y\),
\[
\left|
\int_\Omega\lambda_2u_xv
\right|
\le
\sqrt d\,L_2
\|u_x\|_{L^2(I)}
\|v\|_{L^2(\Omega)}
\]
and therefore, by Young's inequality and \eqref{eq:bathymetric-coercivity},
\[
\left|
\int_\Omega\lambda_2u_xv
\right|
\le
\frac{\mu_v}{2}\|v\|_{H^1(\Omega)}^2
+
b_v\|u_x\|_{L^2(I)}^2.
\]
This proves \eqref{eq:bathymetric-energy-gain}.

Multiplying \eqref{eq:bathymetric-energy-gain} by \(e^{2\theta t}\) and using \(\|v\|_{L^2(\Omega)}\le\|v\|_{H^1(\Omega)}\) gives
\[
\begin{aligned}
\frac12\frac{d}{dt}
\left(e^{2\theta t}\|v(t)\|_{L^2(\Omega)}^2\right)
&
+
\left(\frac{\mu_v}{2}-\theta\right)
e^{2\theta t}\|v(t)\|_{H^1(\Omega)}^2\\
&\quad
\le
b_v e^{2\theta t}\|u_x(t)\|_{L^2(I)}^2.
\end{aligned}
\]
Integrating over \((0,T)\) and discarding the nonnegative terminal term gives
\[
(\mu_v-2\theta)
\|v\|_{L^2_\theta(0,T;H^1(\Omega))}^2
\le
\|v_0\|_{L^2(\Omega)}^2
+
2b_v
\|u_x\|_{L^2_\theta(0,T;L^2(I))}^2.
\]
Taking square roots and using \(\sqrt{a+b}\le\sqrt a+\sqrt b\) proves \eqref{eq:weighted-bathymetric-gain}.
\end{proof}

The preceding estimates yield a weighted bound for the residual term \(\mathfrak d\) in terms of the shoreline gradient.

\begin{proposition}[Weighted residual estimate]
\label{prop:weighted-residual-gain}
Let
\[
0<\theta<\frac{\mu_v}{2}.
\]
Then, for every \(T>0\),
\begin{align}
\|\mathfrak d\|_{L^2_\theta(0,T;X)}
&\le
C_{\mathfrak d0}(\theta)\|v_0\|_{L^2(\Omega)}
\nonumber\\
&\quad+
G_{\mathfrak du}(\theta)
\|u_x\|_{L^2_\theta(0,T;L^2(I))},
\label{eq:weighted-residual-gain}
\end{align}
where
\begin{equation}
\label{eq:Cd0}
C_{\mathfrak d0}(\theta)
=
M_{12}C_{v0}(\theta),
\end{equation}
and
\begin{equation}
\label{eq:Gdu}
G_{\mathfrak du}(\theta)
=
m_\alpha\frac{\ell}{\pi}
+
M_{12}
\sqrt{\frac{2b_v}{\mu_v-2\theta}}.
\end{equation}
\end{proposition}

\begin{proof}
From the definition of \(\mathfrak d\) and Lemma~\ref{lem:bounded_coupling},
\[
\|\mathfrak d(t)\|_X
\le
m_\alpha\|u(t)\|_{L^2(I)}
+
M_{12}\|v(t)\|_{H^1(\Omega)}.
\]
Since \(u(t)\in H_0^1(I)\), for almost every \(t>0\), Poincar\'e's inequality gives
\[
\|u(t)\|_{L^2(I)}
\le
\frac{\ell}{\pi}
\|u_x(t)\|_{L^2(I)},
\]
and therefore
\[
\|\mathfrak d(t)\|_X
\le
m_\alpha\frac{\ell}{\pi}
\|u_x(t)\|_{L^2(I)}
+
M_{12}\|v(t)\|_{H^1(\Omega)}.
\]
Taking the weighted \(L^2\)-norm and using the triangle inequality yields
\begin{align*}
\|\mathfrak d\|_{L^2_\theta(0,T;X)}
&\le
m_\alpha\frac{\ell}{\pi}
\|u_x\|_{L^2_\theta(0,T;L^2(I))}
\\
&\quad+
M_{12}
\|v\|_{L^2_\theta(0,T;H^1(\Omega))}.
\end{align*}
Proposition~\ref{prop:bathymetric-weighted-gain} then gives
\begin{align*}
\|\mathfrak d\|_{L^2_\theta(0,T;X)}
&\le
M_{12}C_{v0}(\theta)
\|v_0\|_{L^2(\Omega)}
\\
&\quad+
m_\alpha\frac{\ell}{\pi}
\|u_x\|_{L^2_\theta(0,T;L^2(I))}
\\
&\quad+
M_{12}
\sqrt{\frac{2b_v}{\mu_v-2\theta}}\,
\|u_x\|_{L^2_\theta(0,T;L^2(I))}.
\end{align*}
The definitions \eqref{eq:Cd0}--\eqref{eq:Gdu} give \eqref{eq:weighted-residual-gain}.
\end{proof}

The preceding estimates define the two weighted gains \(\mathfrak d\mapsto u_x\) and \(u_x\mapsto\mathfrak d\). Their interconnection is analyzed through the following small-gain condition.

\begin{proposition}[Small-gain closure]
\label{prop:small-gain-closure}
Define
\[
\mathcal I_0
=
\|(u_0,v_0)\|_{\mathcal H}
+
\|U_0\|_{L^2(-D,0;X)}.
\]
Assume that, for some
\[
0<\theta<
\min\left\{\omega_F,\frac{\mu_v}{2}\right\},
\]
\begin{equation}
\label{eq:small-gain-condition}
\Gamma(\theta)
:=
G_{u\mathfrak d}(\theta)G_{\mathfrak du}(\theta)
<1.
\end{equation}
Then there exists \(C_\theta>0\), independent of \(T\), such that
\begin{align}
\|u_x\|_{L^2_\theta(0,T;L^2(I))}
&+
\|\mathfrak d\|_{L^2_\theta(0,T;X)}
+
\|v\|_{L^2_\theta(0,T;H^1(\Omega))}\nonumber\\
&+
\|Y\|_{L^2_\theta(0,T;H_0^1(I))}
\le
C_\theta\mathcal I_0.
\label{eq:small-gain-global-estimate}
\end{align}
\end{proposition}

\begin{proof}
Since \(\mathfrak d\in L^2(0,T;X)\) on every finite interval, all preceding weighted estimates apply. For \(T\le D\), the conclusion follows by restricting to \([0,T]\) the corresponding estimates on the fixed interval \([0,D]\); hence assume \(T>D\).

Theorem~\ref{thm:global-wellposedness}, applied on the fixed interval \([0,D]\), and the bound \(e^{2\theta t}\le e^{2\theta D}\) yield
\begin{equation}
\label{eq:initial-delay-interval}
\|u_x\|_{L^2_\theta(0,D;L^2(I))}
\le
C_{D,\theta}\mathcal I_0.
\end{equation}
Moreover, from \eqref{eq:predictor-initial},
\begin{equation}
\label{eq:Y0-initial-bound}
\|Y_0\|_X
\le
C_{Y,D}\mathcal I_0.
\end{equation}
Hence Proposition~\ref{prop:weighted-shoreline-gain} yields
\begin{equation}
\label{eq:ux-complete-gain}
\|u_x\|_{L^2_\theta(0,T;L^2(I))}
\le
C_{u0}(\theta)\mathcal I_0
+
G_{u\mathfrak d}(\theta)
\|\mathfrak d\|_{L^2_\theta(0,T;X)},
\end{equation}
where \(C_{u0}(\theta)>0\) is independent of \(T\).

On the other hand, Proposition~\ref{prop:weighted-residual-gain} gives
\begin{equation}
\label{eq:d-complete-gain}
\|\mathfrak d\|_{L^2_\theta(0,T;X)}
\le
C_{\mathfrak d0}(\theta)\mathcal I_0
+
G_{\mathfrak du}(\theta)
\|u_x\|_{L^2_\theta(0,T;L^2(I))}.
\end{equation}
Combining \eqref{eq:ux-complete-gain} and \eqref{eq:d-complete-gain} yields
\[
\begin{aligned}
\bigl(1-\Gamma(\theta)\bigr)
\|u_x\|_{L^2_\theta(0,T;L^2(I))}
&\le
\left[
C_{u0}(\theta)
+
G_{u\mathfrak d}(\theta)
C_{\mathfrak d0}(\theta)
\right]\mathcal I_0.
\end{aligned}
\]
Since \(\Gamma(\theta)<1\), the shoreline-gradient estimate is uniform in \(T\). Substitution into \eqref{eq:d-complete-gain}, \eqref{eq:weighted-bathymetric-gain}, and \eqref{eq:weighted-Y-H1} proves \eqref{eq:small-gain-global-estimate}.
\end{proof}

The preceding small-gain estimate provides the uniform weighted bounds from which pointwise exponential decay of the full delayed closed-loop state is derived.

\begin{theorem}[Exponential stability of the closed-loop system]
\label{thm:full-exponential-stability}
Let \(K_{\rm cr}\) be given by Proposition~\ref{prop:critical-controllability}, and let \(A_F\) be the exponentially stable operator of Theorem~\ref{thm:finite-infinite-lifting}. Assume \eqref{eq:bathymetric-coercivity} and suppose that, for some
\[
0<\theta<
\min\left\{\omega_F,\frac{\mu_v}{2}\right\},
\]
the small-gain condition \eqref{eq:small-gain-condition} holds.
Then the predictor feedback
\[
U(t)=K_{\rm cr}P_{\rm cr}Y(t)
\]
exponentially stabilizes the delayed coupled system. More precisely, there exists \(M\ge1\) such that
\begin{equation}
\label{eq:final-exponential-estimate}
\|(u(t),v(t),z(t))\|_{\mathcal X}
\le
Me^{-\theta t}\mathcal I_0,
\qquad t\ge0,
\end{equation}
where
\[
\mathcal X
=
\mathcal H\times\mathcal Z,
\qquad
\mathcal Z=L^2((0,D);X).
\]
\end{theorem}

\begin{proof}
Letting \(T\to\infty\) in Proposition~\ref{prop:small-gain-closure} yields
\begin{equation}
\label{eq:global-weighted-estimate}
\begin{aligned}
\|u_x\|_{L^2_\theta(0,\infty;L^2(I))}
&+
\|\mathfrak d\|_{L^2_\theta(0,\infty;X)}
+
\|v\|_{L^2_\theta(0,\infty;H^1(\Omega))}\\
&+
\|Y\|_{L^2_\theta(0,\infty;H_0^1(I))}
\le
C_\theta\mathcal I_0.
\end{aligned}
\end{equation}

For the predictor, variation of constants gives
\[
Y(t)
=
e^{A_Ft}Y_0
+
\int_0^t
e^{A_F(t-s)}e^{A_nD}\mathfrak d(s)\,ds.
\]
Since \(\theta<\omega_F\), Cauchy--Schwarz gives
\[
\begin{aligned}
\left\|
\int_0^t
e^{A_F(t-s)}e^{A_nD}\mathfrak d(s)\,ds
\right\|_X
&\le
\frac{M_F\|e^{A_nD}\|_{\mathcal L(X)}}
{\sqrt{2(\omega_F-\theta)}}
e^{-\theta t}\\
&\quad\times
\|\mathfrak d\|_{L^2_\theta(0,\infty;X)}.
\end{aligned}
\]
Hence, using \eqref{eq:Y0-initial-bound} and \eqref{eq:global-weighted-estimate},
\begin{equation}
\label{eq:Y-pointwise-decay}
\|Y(t)\|_X
\le
C_Ye^{-\theta t}\mathcal I_0.
\end{equation}

For \(t\ge D\), \eqref{eq:u-predictor-weighted} and \eqref{eq:Y-pointwise-decay} yield
\[
\|u(t)\|_X
\le
C_Ye^{\theta D}e^{-\theta t}\mathcal I_0
+
M_{n,D}
\int_{t-D}^{t}\|\mathfrak d(s)\|_X\,ds,
\]
where
\[
M_{n,D}
=
\sup_{0\le r\le D}
\|e^{A_nr}\|_{\mathcal L(X)}.
\]
Moreover,
\[
\int_{t-D}^{t}\|\mathfrak d(s)\|_X\,ds
\le
e^{-\theta(t-D)}
\sqrt D\,
\|\mathfrak d\|_{L^2_\theta(0,\infty;X)}.
\]
Thus,
\begin{equation}
\label{eq:u-pointwise-decay}
\|u(t)\|_X
\le
C_ue^{-\theta t}\mathcal I_0,
\qquad t\ge D.
\end{equation}

For the bathymetric component, \eqref{eq:bathymetric-energy-gain} and \(\|v\|_{L^2(\Omega)}\le\|v\|_{H^1(\Omega)}\) imply
\[
\frac{d}{dt}\|v(t)\|_{L^2(\Omega)}^2
+
\mu_v\|v(t)\|_{L^2(\Omega)}^2
\le
2b_v\|u_x(t)\|_{L^2(I)}^2.
\]
Hence
\[
\begin{aligned}
\|v(t)\|_{L^2(\Omega)}^2
&\le
e^{-\mu_vt}\|v_0\|_{L^2(\Omega)}^2\\
&\quad+
2b_v
\int_0^t
e^{-\mu_v(t-s)}
\|u_x(s)\|_{L^2(I)}^2\,ds.
\end{aligned}
\]
Since \(2\theta<\mu_v\),
\[
\begin{aligned}
\int_0^t
e^{-\mu_v(t-s)}
\|u_x(s)\|_{L^2(I)}^2\,ds
&\le
e^{-2\theta t}
\|u_x\|_{L^2_\theta(0,\infty;L^2(I))}^2.
\end{aligned}
\]
Together with \eqref{eq:global-weighted-estimate}, this gives
\begin{equation}
\label{eq:v-pointwise-decay}
\|v(t)\|_{L^2(\Omega)}
\le
C_ve^{-\theta t}\mathcal I_0.
\end{equation}

Finally,
\[
z(s,t)=U(t-D+s),
\qquad
U(t)=K_{\rm cr}P_{\rm cr}Y(t).
\]
For \(t\ge D\), \eqref{eq:Y-pointwise-decay} gives
\[
\|z(t)\|_{\mathcal Z}
\le
C_ze^{-\theta t}\mathcal I_0.
\]
On the fixed interval \(0\le t<D\), the finite-time estimate of Theorem~\ref{thm:global-wellposedness} controls \(u\), \(v\), and \(z\), and the factor \(e^{\theta D}\) is absorbed into the constant.
Combining \eqref{eq:u-pointwise-decay}, \eqref{eq:v-pointwise-decay}, the transport-state estimate, and the finite-time bound on \([0,D]\) proves \eqref{eq:final-exponential-estimate}.
\end{proof}

\section{Numerical Illustration}
\label{sec:numerics}

The simulations examine, on normalized scales, the effect of delayed shoreline actuation and its predictor compensation on the coupled shoreline--bathymetry dynamics. Particular attention is given to the critical shoreline modes and to the unactuated bathymetric response.

\subsection{Numerical Setting and Critical-Mode Design}
\label{subsec:numerical-setting}

The computational domains are defined by
\[
\ell=200,
\qquad
d=100.
\]
The shoreline domain is discretized by centered finite differences on \(N_x=180\) interior points with homogeneous Dirichlet conditions. The bathymetric subsystem is discretized on a tensor grid with \(N_{x,2}=90\) and \(N_y=50\), subject to homogeneous Neumann conditions.
Diffusion is treated implicitly, whereas reaction and coupling terms are evaluated explicitly. The simulation parameters are
\[
T=2000,
\qquad
\Delta t=0.05,
\]
and the displayed coordinates are normalized as \(x/\ell\), \(y/d\), and \(t/T\).

The principal coefficients are
\[
\epsilon_1=3\times10^{-3},
\qquad
\epsilon_2=\bar\epsilon_2=4\times10^{-2},
\qquad
\bar\alpha_1=1.5\times10^{-5}.
\]
The spatially and temporally varying coefficients used in the simulations are
\[
g(x)
=
0.02
+
0.10\,\mathbf 1_{[0.35\ell,\,0.65\ell]}(x),
\]
\[
\widetilde\alpha_1(x,t)
=
3\times10^{-6}
\cos\left(\frac{2\pi x}{\ell}\right)
\sin\left(\frac{2\pi t}{300}\right),
\]
\[
\lambda_1(x,y)
=
3\times10^{-2}
\left(
1+0.15\cos\frac{2\pi x}{\ell}
\right)
\exp\left(-\frac{y}{1.2d}\right),
\]
\[
\lambda_2(x)
=
2.5\times10^{-2}
\left(
1+0.10\sin\frac{2\pi x}{\ell}
\right),
\]
and
\[
\alpha_2(x,y)
=
-8\times10^{-4}
\left(
1+0.10\cos\frac{2\pi x}{\ell}
\right)
\left(
1+\frac{1}{20}\sin\frac{\pi y}{d}
\right).
\]

For the resulting semi-discrete bathymetric operator,
\[
\mu_{v,h}\ge 7.2005\times10^{-4}.
\]
Consequently, the discrete counterpart of the restriction on the exponential weight admits the range
\[
0<\theta<3.60025\times10^{-4}.
\]
This provides a consistency check with the bathymetric coercivity requirement, without constituting a numerical verification of the continuous small-gain condition.
The numerical energies defined below are instead unweighted discrete \(L^2\) diagnostics.

The initial conditions are
\[
u_0(x)
=
0.50\sin\frac{\pi x}{\ell}
+
0.20\sin\frac{2\pi x}{\ell}
+
0.10\sin\frac{4\pi x}{\ell},
\]
and
\[
v_0(x,y)
=
0.20\cos\frac{\pi x}{\ell}\cos\frac{\pi y}{d}
+
0.08\cos\frac{2\pi x}{\ell}\cos\frac{2\pi y}{d},
\]
with zero initial input history on \([-D,0]\).

The nominal semi-discrete shoreline operator has four critical eigenvalues,
\[
10^{-4}
\begin{bmatrix}
0.1426 & 0.1204 & 0.0834 & 0.0316
\end{bmatrix},
\]
and the corresponding closed-loop poles are assigned to
\[
-0.004,\qquad -0.005,\qquad -0.006,\qquad -0.007.
\]

For comparison, define
\begin{equation}
E_h(t)
=
E_{u,h}(t)+E_{v,h}(t)+E_{D,h}(t),
\label{eq:numerical-total-energy}
\end{equation}
where
\[
E_{u,h}(t)=\|u_h(t)\|_h^2,
\qquad
E_{v,h}(t)=\|v_h(t)\|_h^2,
\]
and
\[
E_{D,h}(t)
\approx
\int_0^D\|z_h(s,t)\|_h^2\,ds.
\]
Here, the discrete norms include the corresponding spatial quadrature weights. In open loop, \(E_{D,h}\equiv0\). The critical shoreline diagnostic is
\[
E_{{\rm cr},h}(t)
=
\|P_{{\rm cr},h}u_h(t)\|_h^2.
\]
These quantities are performance indicators and are not identified with the weighted functional used in the stability proof.

\subsection{Delay Robustness and Representative Delay}
\label{subsec:delay-robustness}

Using the numerical setting of Section~\ref{subsec:numerical-setting}, a coarse delay sweep over \(D\in[10,210]\), followed by refinement near the loss-of-decay region, compares uncompensated delayed feedback with predictor compensation.
For each delay, finite-horizon terminal-envelope slopes, resurgence indicators, and final-energy ratios are evaluated.

For normalization,
\[
D_{\rm crit}^{\rm naive}=224.234
\]
denotes the delay margin computed for the isolated semi-discrete critical subsystem under uncompensated delayed feedback. It is used only as a reference scale and is not a stability margin for the full shoreline--bathymetry system.

Figure~\ref{fig:delay-robustness} shows that increasing the uncompensated delay progressively weakens the decay of the terminal total-energy envelope, whose estimated slope loses negativity at approximately
\[
D_{\rm num}\simeq143.75.
\]
This transition is consistent with the progressive loss of effectiveness of a feedback signal computed from a shoreline state that is increasingly out of phase with the state reached at the actuation time. Accordingly, \(D_{\rm num}\) is interpreted only as an empirical finite-horizon decay-loss threshold and not as a stability margin of the continuous closed-loop system.

The representative value
\[
D^\star=120,
\qquad
\frac{D^\star}{D_{\rm num}}\simeq0.835,
\qquad
\frac{D^\star}{D_{\rm crit}^{\rm naive}}\simeq0.535,
\]
is deliberately selected below the observed decay-loss threshold. It therefore remains within the empirically decaying uncompensated regime while making the delay-induced degradation sufficiently pronounced to assess the effect of predictor compensation. At the same delay, the predictor yields negative terminal-envelope slopes for both the total and critical energies,
\[
-1.6645\times10^{-3},
\qquad
-2.1317\times10^{-3},
\]
respectively. The simultaneous decay of these diagnostics shows that the predictor-based improvement is not confined to the final total-energy value: it is observed in the critical shoreline component targeted by the feedback and in the coupled shoreline--bathymetry response.

\begin{figure}[!ht]
\centering
\includegraphics[width=\linewidth]{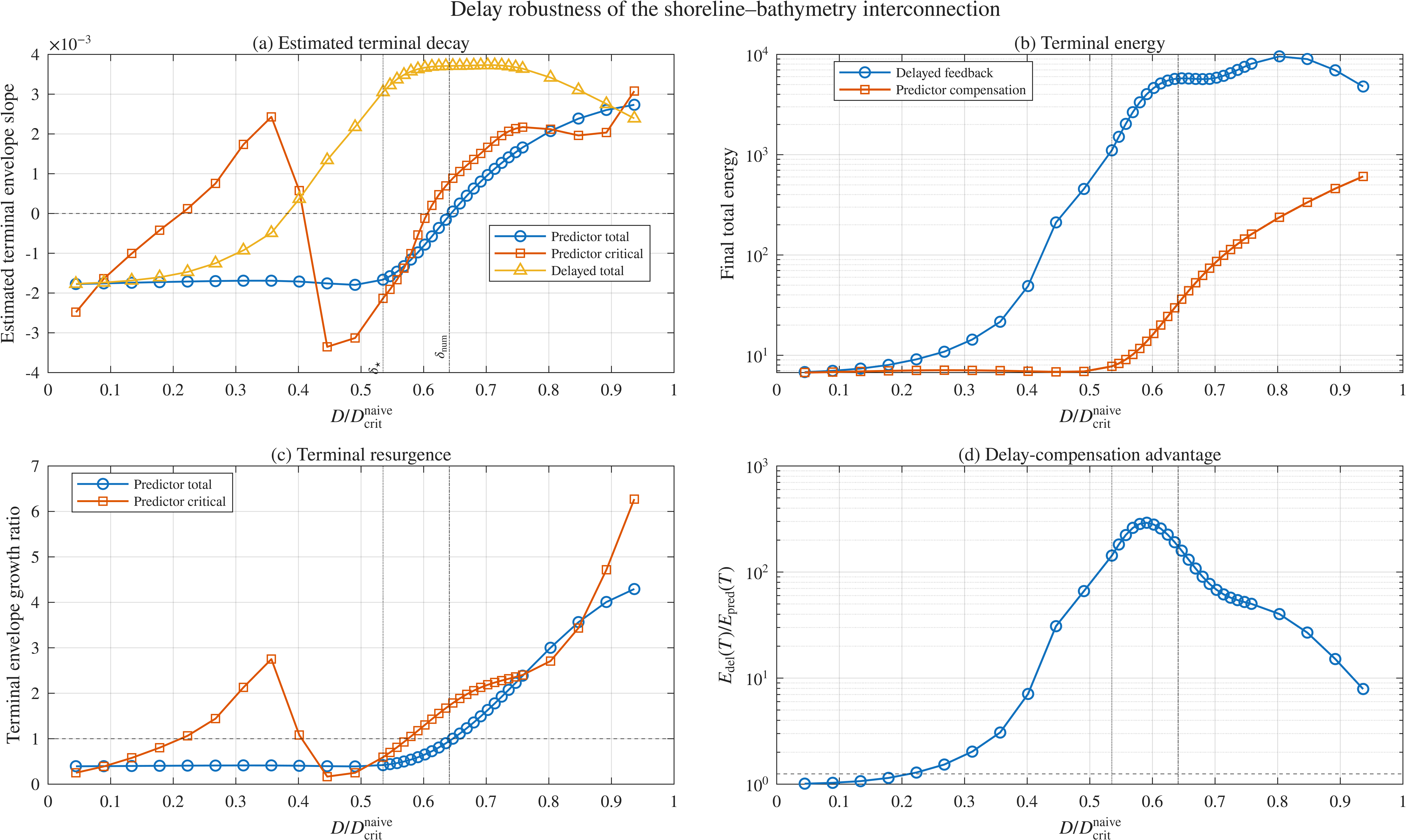}
\caption{Delay sensitivity of the coupled dynamics. The vertical markers denote \(D^\star\) and the empirical decay-loss threshold \(D_{\rm num}\); \(D_{\rm crit}^{\rm naive}\) refers only to the isolated semi-discrete critical subsystem.}
\label{fig:delay-robustness}
\end{figure}

\subsection{Closed-Loop Shoreline--Bathymetry Response}
\label{subsec:closed-loop-numerics}

At \(D=D^\star\), Fig.~\ref{fig:closed-loop-performance} compares open loop, uncompensated delayed feedback, and predictor compensation.
The comparison reveals two distinct effects of the actuator delay.
Without compensation, the feedback initially acts on the critical shoreline dynamics but eventually produces a strong energy resurgence.
Predictor compensation suppresses this degradation and preserves the decay of the critical component targeted by the feedback.

At \(t=T\),
\[
E_h^{\rm open}(T)=13.794,
\;
E_h^{\rm del}(T)=1083.2,
\;
E_h^{\rm pred}(T)=7.716.
\]
Predictor compensation therefore reduces the final total energy by \(99.29\%\) relative to uncompensated delayed feedback and by \(44.06\%\) relative to open loop. More importantly, on the critical shoreline subspace,
\[
E_{{\rm cr},h}^{\rm open}(T)=6.7033,
\;
E_{{\rm cr},h}^{\rm del}(T)=81.477,
\;
E_{{\rm cr},h}^{\rm pred}(T)=0.118.
\]
The pronounced reduction of the critical energy shows that the improvement is not confined to the terminal energy of the full discretized system. It is also observed on the finite-dimensional shoreline component targeted by the feedback, consistently with the role of predictor compensation in reducing the adverse effect of the actuator delay on the controlled critical dynamics.

\begin{figure}[!ht]
\centering
\includegraphics[width=\linewidth]{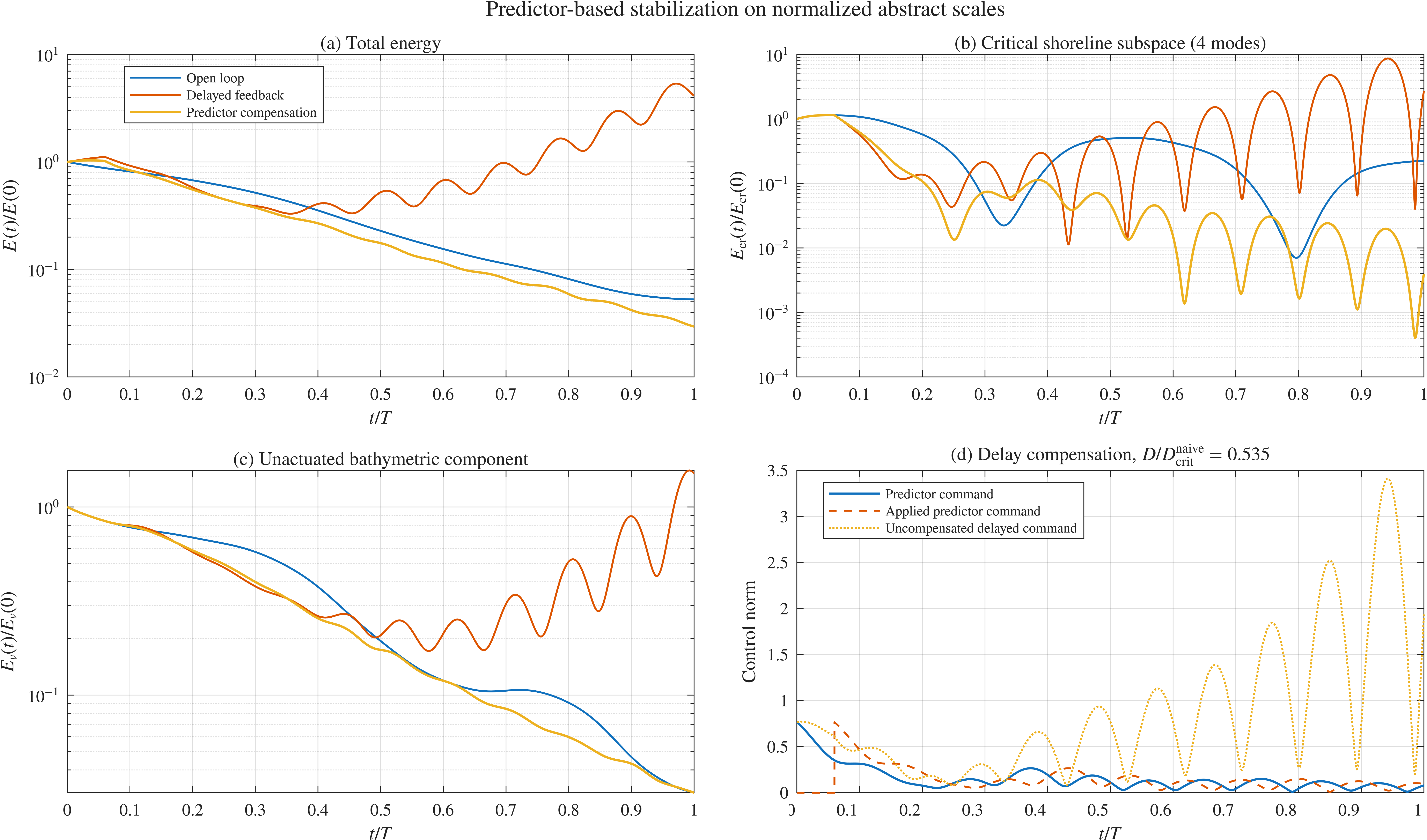}
\caption{Closed-loop diagnostics at \(D=D^\star\): total, critical shoreline, and bathymetric energies, together with the control signals.}
\label{fig:closed-loop-performance}
\end{figure}

Figure~\ref{fig:shoreline-dynamics} provides the spatial counterpart of these energy diagnostics. The uncompensated delayed feedback generates pronounced oscillatory amplification, showing that the deterioration is not confined to a terminal-energy indicator but develops throughout the shoreline dynamics. In contrast, predictor compensation maintains a substantially lower-amplitude response over the space--time domain. This behavior is consistent with the role of the predictor in compensating the temporal mismatch between command generation and delayed actuation.

\begin{figure}[!ht]
\centering
\includegraphics[width=\linewidth]{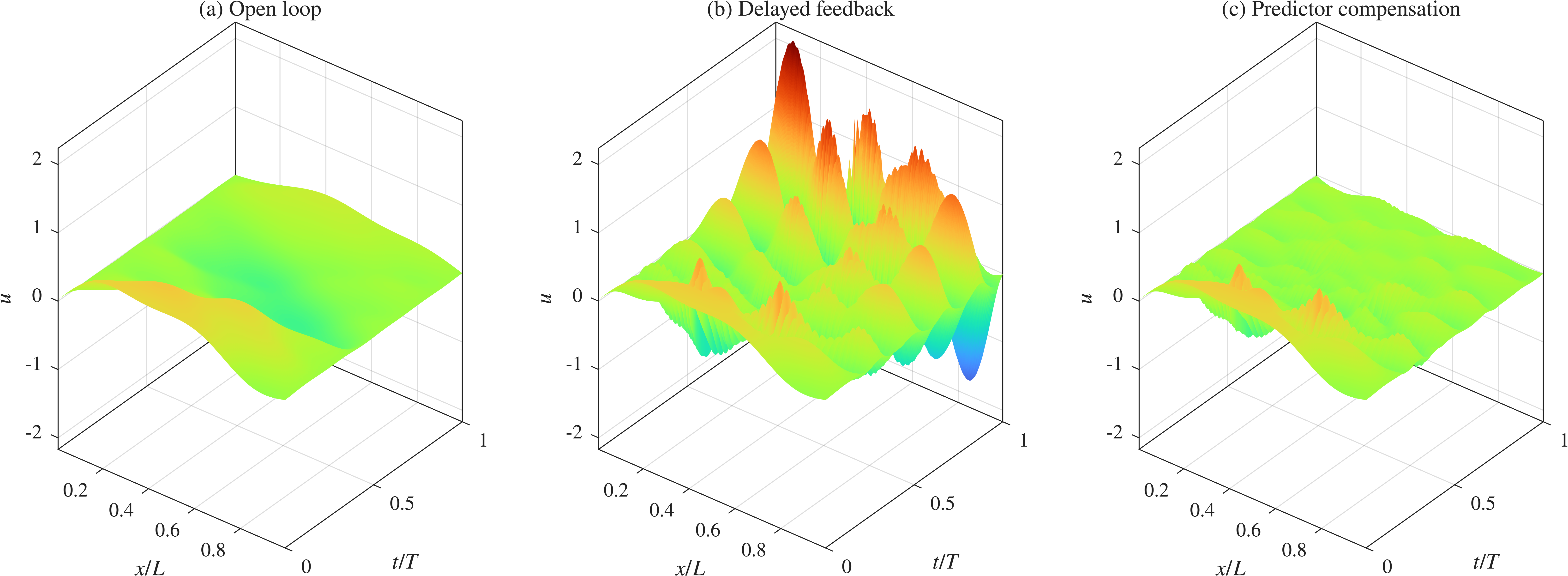}
\caption{Shoreline dynamics at \(D=D^\star\): (a) open loop, (b) uncompensated delayed feedback, and (c) predictor compensation. Identical scales are used in all panels.}
\label{fig:shoreline-dynamics}
\end{figure}

The bathymetric subsystem is not directly actuated, making Fig.~\ref{fig:bathymetric-dynamics} particularly informative about the coupled dynamics. The large shoreline amplification generated by the uncompensated delayed feedback propagates to the bathymetric component through the hetero-dimensional coupling. At the final time,
\[
E_{v,h}^{\rm open}(T)=7.0904,
\;
E_{v,h}^{\rm del}(T)=348.68,
\;
E_{v,h}^{\rm pred}(T)=7.0227.
\]
Predictor compensation reduces the bathymetric energy by \(97.99\%\) relative to uncompensated delayed feedback and returns it close to its open-loop level. The significance of this result is therefore not direct bathymetric regulation. Rather, predictor compensation limits the delay-induced shoreline amplification and prevents its propagation through the coupling to the unactuated bathymetric subsystem. This behavior is qualitatively consistent with the coupled stability mechanism captured by the small-gain analysis.

\begin{figure}[!ht]
\centering
\includegraphics[width=\linewidth]{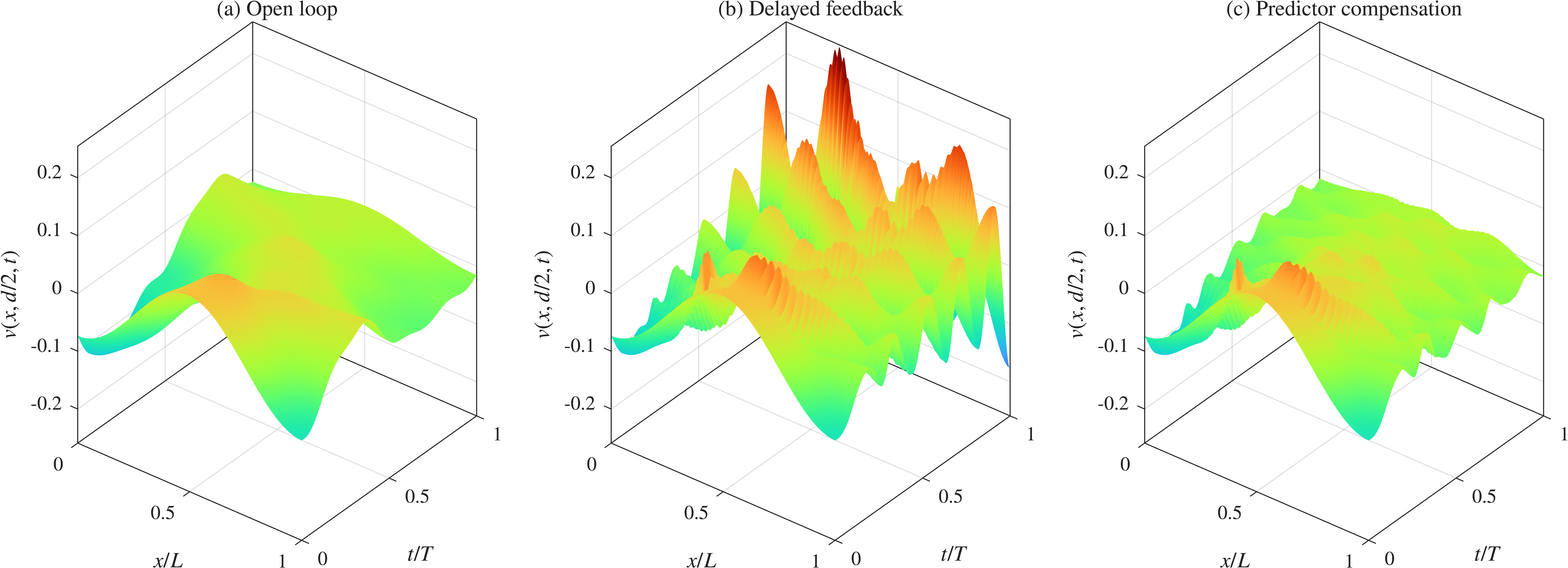}
\caption{Bathymetric centerline \(v(x,d/2,t)\) at \(D=D^\star\): (a) open loop, (b) uncompensated delayed feedback, and (c) predictor compensation. Identical scales are used in all panels.}
\label{fig:bathymetric-dynamics}
\end{figure}

\section{Conclusion}

Exponential stabilization has been established for a delayed hetero-dimensional shoreline--bathymetry system through shoreline actuation alone. The feedback is constructed on the finite-dimensional critical shoreline subspace and lifted to the full shoreline dynamics, while predictor compensation removes the explicit input delay from the nominal predictor dynamics. The remaining shoreline--bathymetry interaction is quantified through weighted input--output estimates, and predictor-based shoreline stabilization is combined with bathymetric coercivity through a small-gain condition to establish exponential stability of the full delayed closed-loop system.

The numerical results illustrate the mechanisms captured by the analysis. Uncompensated delayed feedback progressively loses effectiveness and can strongly amplify both the shoreline response and the indirectly coupled bathymetric dynamics. Predictor compensation restores decay of the critical shoreline component and prevents this delay-induced amplification from propagating to the unactuated bathymetric subsystem. Robustness with respect to uncertain environmental forcing and extensions to more detailed morphodynamic configurations remain natural directions for further study.

\bibliographystyle{unsrt}

\end{document}